\documentclass[12pt,reqno]{amsart}
\usepackage[utf8]{inputenc}
\usepackage[english]{babel}
\usepackage{amsmath}
\usepackage{amsfonts}
\usepackage{amssymb}
\usepackage{amsthm}
\usepackage{hyperref}
\usepackage{cleveref}
\usepackage{color}
\usepackage{mathrsfs}

\makeatletter
\newcounter{savesection}
\newcounter{apdxsection}
\renewcommand\appendix{\par
	\setcounter{savesection}{\value{section}}%
	\setcounter{section}{\value{apdxsection}}%
	\setcounter{subsection}{0}
	\gdef\thesection{Appendix \@Alph\c@section}}

\newcommand\unappendix{\par
	\setcounter{apdxsection}{\value{section}}%
	\setcounter{section}{\value{savesection}}%
	\setcounter{subsection}{0}%
	\gdef\thesection{\@arabic\c@section}}
\makeatother

\newtheorem*{theorem*}{Theorem}
\newtheorem{theorem}{Theorem}[section]
\newtheorem{corollary}{Corollary}[theorem]
\newtheorem{lemma}[theorem]{Lemma}
\newtheorem{remark}[theorem]{Remark}
\newtheorem*{remark*}{Remark}
\newtheorem*{lemma*}{Lemma}

\newcommand{\norm}[1]{\left\lVert#1\right\rVert}

\begin{document}
\title[Improved F. D. Results for the Caputo Time-Fractional Diff. Equation]{Improved Finite Difference Results for the Caputo Time-Fractional Diffusion Equation}
\author[Wesley Davis]{W. Davis}
\address{Department of Mathematics and Statistics, Old Dominion University, Norfolk, VA.}
\email{wdavi002@odu.edu}
\author[Richard Noren]{R. Noren}
\address{Department of Mathematics and Statistics, Old Dominion University, Norfolk, VA.}
\email{rnoren@odu.edu}
\author[Ke Shi]{K. Shi}
\address{Department of Mathematics and Statistics, Old Dominion University, Norfolk, VA.}
\email{kshi@odu.edu}

\begin{abstract}
We begin with a treatment of the Caputo time-fractional diffusion equation, by using the Laplace transform, to obtain a Volterra intego-differential equation where we may examine the weakly singular nature of this convolution kernel.We examine this new equation and utilize a numerical scheme that is derived in parallel to the L1-method for the time variable and a usual fourth order approximation in the spatial variable. The main method derived in this paper has a rate of convergence of $O(k^{2}+h^4)$ for $u(x,t) \in C^6(\Omega)\times C^2[0,T]$, which improves previous estimates by a factor of $k^{\alpha}$. We also present a novel alternative method for a first order approximation in time, which allows us to relax our regularity assumption to $u(x,t) \in C^6(\Omega)\times C^1[0,T]$, while exhibiting order of convergence slightly less than $O(k^{1+\alpha})$ in time. This allows for a much wider class of functions to be analyzed which was previously not possible under the L1-method. We present numerical examples demonstrating these results and discuss future improvements and implications by using these techniques.
\end{abstract}
\maketitle
\section{Introduction}
Fractional differential equations have been of great interest to various fields in physics, engineering, and mathematics over the past several decades, as seen in [10,11] and many others. Many applications of fractional diffusion equations are studied due to their physical applications, we refer to [10-16] for a small survey of relevant and related works. In their 2014 article \cite{Zhang2014}, Zhang et al. established a numerical scheme for the one-dimensional time-fractional order diffusion equation with initial and boundary conditions
\begin{equation} \label{eq:fraceq}
\mathcal{D}_t^{\alpha} u({\bf x}, t) = \dfrac{\partial^2}{\partial x^2} u({\bf x}, t) + f({\bf x}, t),\ {\bf x}\in \Omega, \ t \in [0, T], 
\end{equation}
\begin{equation*}\label{ibc}
u(x, 0) = \phi (x),\ x\in [0, 1] \text{ and } u(0, t) = u(1, t) = 0,
\end{equation*}
with $\alpha \in (0,1)$ order Caputo fractional time derivative defined by
\begin{equation*}\label{fracder}
\mathcal{D}_t^{\alpha} u({\bf x}, t) = \frac{1}{\Gamma(1-\alpha)} \int_0^t \frac{\partial u({{\bf x}}, s)}{\partial s} (t-s)^{-\alpha}\,ds ,
\end{equation*}
where $\Gamma(x) = \int_0^{\infty} e^{-t} t^{x-1} \,dt$. Various authors have placed various hypotheses on $\phi$ and $f$ in their analysis, see [1,2,15,16]. This problem was solved numerically in \cite{Zhang2014} on the domain $[0,1]\times [0,T]$ with numerical accuracy of order $O(k^{2-\alpha}+h^4)$ by application of a 4th order spatial and a 2nd order time scheme, where k denotes the time mesh size and h denotes the space mesh size, with a constant that depends on $T^{\alpha}$. The 2nd order time scheme is the so-called L1-method, which has been studied extensively in previous works, see \cite{Zhang2014} for further discussion.  In section 2 we will transform \eqref{eq:fraceq} into its equivalent form
\begin{equation}\label{eq:VE}
u(x, t) = \phi(x) + \bigg(a_{1-\alpha} * \bigg(\dfrac{\partial^2 u}{\partial x^2}+f\bigg)\bigg)(x, t),
\end{equation}
where $\displaystyle{a_{1-\alpha}(t) = \frac{t^{\alpha-1}}{\Gamma(\alpha)}}$, by application of the Laplace transform and $*$ denotes convolution as defined in Section 2. We will use the same 4th order discrete space operator as in \cite{Zhang2014}, and we construct two time discretizations for functions $g(t) \in C^1[0,T]$ and for $g(t) \in C^2[0,T]$. {{Thanks to the use of Laplace transform, we are able to relax the regularity assumption to $g(t) \in C^1[0,T]$, under such weaker regularity setting, our analysis shows that the order of convergence is $\mathcal{O}(k + h^4)$ but numerical experences show a better convergence rate as $\mathcal{O}(k^{1 + \alpha} + h^4)$. In addition,}} with the same regularity assumptions as in \cite{Zhang2014} ($g(t) \in C^2[0,T]$), {{we can modify the scheme such that it provides an optimal convergence rate as}} $O(k^2+h^4)$. 
\\
\indent Existence, uniqueness, and monotonicity results were established in \cite{Friedman1969} by A. Friedman for the solution of a generalization of equation \eqref{eq:VE}, see Corollary 1 of [2, p.143]. 
Applications are referenced as well in [2, p.146-147]. More recently, M. Stynes et al were able to obtain existence and uniqueness for the solution of a generalization of equation \eqref{eq:VE} in \cite{Stynes2017}, see Theorem 2.1 of \cite{Stynes2017} for further discussion.
\\
\indent The remainder of the paper is organized as follows. Section 2 presents the numerical preliminaries and presents the existence and uniqueness of a solution to this newly transformed equation. Section 3 defines the numerical schemes and establishes the necessary lemmas for our a priori error estimates. Section 4 contains the statements of our main theorems presented in this paper guaranteeing the convergence and stability of our method. Section 5 presents some numerical examples illustrating our results, where we observe order $O(k^2)$ in time convergence for the $C^2[0,T]$ scheme, and slightly less than order $O(k^{1+\alpha})$ in time convergence for the $C^1[0,T]$ scheme.  Finally, we conclude our findings in Section 6. 

\section{Preliminaries}

We will begin by showing that \eqref{eq:fraceq} and \eqref{eq:VE} are equivalent by application of the Laplace transform under the hypotheses of Theorem A from \cite{Stynes2017}, which we state below.  

Let $\{(\lambda_i,\psi_i): i=1,2,... \}$ be the eigenvalues and eigenfunctions for the Sturm-Liouville two-point boundary value problem
\begin{equation*}
\mathcal{L}\psi_i = -p\dfrac{\partial^2 \psi_i }{\partial x^2} +c\psi_i = \lambda_i \psi_i \ \text{ on } (0,1), \ \ \ \psi_i(0)=\psi_i(1)=0,
\end{equation*}
where the eigenfunctions are normalized by requiring  $\norm{\psi_i}_2=1$ for all i. Define the fractional power $\mathcal{L}^{\gamma}$ of the operator $\mathcal{L}$ for each $\gamma \in \mathbb{R}$ with corresponding domain
\begin{equation*}
D(\mathcal{L}^{\gamma}) = \left\lbrace g\in H^2_0(0,1): \sum_{i=1}^{\infty}\lambda_i^{2\gamma} \vert (g,\psi_i) \vert <\infty  \right\rbrace \subset L^2(0,1).
\end{equation*}
Further, we will use the Sobolev space norm 
\begin{equation*}
\norm{g}_{\mathcal{L}^{\gamma}}= \left(\sum_{i=1}^{\infty}\lambda_i^{2\gamma} \vert (g,\psi_i) \vert \right)^{1/2}, \ \mbox{for all } g\in  D(\mathcal{L}^{\gamma}).
\end{equation*}
\begin{theorem*}[\bf{A [3, p.1061]}]
Let $\phi \in D(\mathcal{L}^{5/2})$, $f(\boldsymbol{\cdot},t) \in D(\mathcal{L}^{5/2})$, $f_{t}(\boldsymbol{\cdot},t) \in D(\mathcal{L}^{5/2})$, and $f_{tt}(\boldsymbol{\cdot},t) \in D(\mathcal{L}^{5/2})$ for each $t \in (0,T]$ with 
\begin{equation*}
\norm{f(\boldsymbol{\cdot},t)}_{\mathcal{L}^{5/2}} + \norm{f_{t}(\boldsymbol{\cdot},t)}_{\mathcal{L}^{1/2}}+ t^{\rho}\norm{f_{tt}(\boldsymbol{\cdot},t)}_{\mathcal{L}^{1/2}} \leq C_{1}
\end{equation*}
for all $t \in (0,T]$ and some constant $\rho < 1$ where $C_{1}$ is a constant independent of t. Then, \eqref{eq:fraceq} has a unique solution u that satisfies the initial and boundary conditions pointwise, and there exists a constant C such that
\begin{align}
\bigg|\dfrac{d^{k}u}{dx^{k}}\bigg| &\leq C  \text{   for k=0,1,2,3,4}\\
\bigg|\dfrac{d^{l}u}{dt^{l}}\bigg| &\leq C(1+t^{\alpha-l})   \text{   for l=0,1,2}.
\end{align} 
\end{theorem*}
\begin{lemma}
Assume the hypotheses of Theorem A. 
Then the function $u=u(x,t)$ satisfies  \eqref{eq:fraceq} if and only if it satisfies \eqref{eq:VE}.
\end{lemma}
\begin{proof}
We use the convolution theorem (see Chapter 6, Section 1.3 of [4, p.135] )
\begin{equation*}
\mathscr{L}(a*b) = \mathscr{L}(a)\mathscr{L}(b) \text{ if } (a*b)(t)=\int_0^t a(t-s)b(s)\,ds
\end{equation*}
and the facts 
\begin{equation*}
\mathscr{L}(a_\alpha)(z) = z^{\alpha-1} \text{ and } \mathscr{L}(h')(z)=z\mathscr{L}(h)(z)-h(0)
\end{equation*}
to obtain
\begin{equation*}
\mathscr{L}(\mathcal{D}_t^{\alpha} u(x, t)) = (z\mathscr{L}(u(x, \cdot))(z)-\phi(x))z^{\alpha-1}
\end{equation*}
Applying the Laplace transform to equation \eqref{eq:fraceq} we obtain after some algebra,
\begin{align*}
(z\mathscr{L}(u(x, \cdot))(z)-&\phi(x))z^{\alpha-1} = \bigg[\mathscr{L}\bigg(\dfrac{\partial^2 u}{\partial x^2}(x,\cdot)\bigg) +\mathscr{L}(f(x,\cdot))(z)\bigg]\\
\mathscr{L}(u(x, \cdot ))(z) &= z^{-1}\phi(x) + z^{-\alpha}\bigg[\mathscr{L}\bigg(\dfrac{\partial^2 u}{\partial x^2}(x,\cdot)\bigg) + \mathscr{L}(f(x,\cdot))(z)\bigg].
\end{align*}
By inverting the Laplace transform, we get the equivalent Volterra integral equation
\begin{equation*}
u(x, t) = \phi(x) + a_{1-\alpha} * \left(\dfrac{\partial^2 u}{\partial x^2}+f\right)(x, t).
\end{equation*}
Since the steps are reversible and our formal manipulations are valid by Theorem A, then the result follows.
\end{proof}
\begin{remark}
The manipulations in the prior lemma use the assumptions from Theorem A in order to guarantee our a priori estimates that are derived in sections 3 and 4. We note that a similar existence and uniqueness theorem can be derived under the assumptions detailed in \cite{Friedman1969}, but the a regularity of the solution that results is insufficient for our finite difference methods.
\end{remark}
With the equivalence established between \eqref{eq:fraceq} and \eqref{eq:VE}, we next provide the finite difference schemes that are used and the resulting a priori error estimates in the following sections. The existence and uniqueness of a solution to \eqref{eq:VE} is presented in Appendix A. 
 We now examine the consistency, stability, and convergence of multiple numerical schemes for \eqref{eq:VE} based on the regularity of the solution in the time-variable.
\section{Fully Discretized Numerical Schemes}
In [1, 15, 16], a fully discrete scheme was derived for the L1-method in the time variable and analyzed as such. By utilizing the Laplace transform, we are able to derive an equation with a different integral kernel than the fractional derivative operator as defined before. Therefore, we will derive two convergent numerical schemes for this newly transformed equation for both a first and second-order approximation to \eqref{eq:VE} in time. The schemes are defined by the degree of regularity that will be assumed, therefore we will construct a first-order accurate scheme for functions that are $C^1[0,T]$ in time and a second-order accurate scheme for functions that are $C^2[0,T]$ in time. From there, we will utilize a spatial operator that was defined in \cite{Zhang2014} which is fourth-order accurate in the spatial variable and hence we will arrive at the fully discrete equations. \\
\indent We will use the notations and state key results from \cite{Zhang2014} that extend to our work.
Divide the time interval [0,T] into N intervals where $0 = t_0 < t_1 < ... < t_N = T$. Denote the time steps as
\begin{equation*}
\tau_n = t_n - t_{n-1}, \ \ 1\leq  n \leq N,
\end{equation*}
and the mesh of the partition
\begin{equation*}
\tau_{max} = \max_{1\leq j \leq N}\tau_j.
\end{equation*}
We will derive the numerical results for any temporal mesh provided, see Theorems 3.1, 3.2, 4.3, and 4.6 for these results. Having established the unique solution to \eqref{eq:VE} in section 2.1, we shall denote the grid function by
\begin{center}
$v = \{v_{i} : \  0\leq i\leq M \}$, where $M>0$, $h=\dfrac{1}{M}$, and $x_{i} = ih$,\\
\end{center}
and the grid operator
\begin{align} \label{Gridop}
\mathcal{H}_{h}v_{i} = \begin{cases} 
			 \dfrac{1}{12}(v_{i+1}+10v_{i}+v_{i-1}), & 1\leq i \leq M-1, \\ 
			 v_{i},  & i=0 $ or $ i=M.
			\end{cases}
\end{align}
By applying $\mathcal{H}_{h}$ to equation \eqref{eq:VE} we see that when $i=0 $ or $i=M$,
\begin{align*}
\mathcal{H}_{h}u(x_{i},t_{n}) &= u(x_{i},0) + \dfrac{1}{\Gamma(\alpha)}\int_0^{t_{n}} (t_{n}-s)^{\alpha -1}\left(\dfrac{\partial^2 u}{\partial x^2}(x_{i},s) + f(x_{i},s)\right)\,ds,
\end{align*}
and when $1\leq i \leq M-1$,
\begin{align} \label{DifferenceEQ}
\mathcal{H}_{h}&u(x_{i},t_{n}) = \dfrac{1}{12}u(x_{i-1},t_{n})+\dfrac{10}{12}u(x_{i},t_{n})+\dfrac{1}{12}u(x_{i+1},t_{n}) \nonumber \\ 
&=  \mathcal{H}_{h}\left[u(x_{i},0) + \dfrac{1}{\Gamma(\alpha)}\int_0^{t_{n}} (t_{n}-s)^{\alpha -1}\left(\dfrac{\partial^2 u}{\partial x^2}(x_{i},s) + f(x_{i},s)\right)\,ds\right] \nonumber \\
&= \mathcal{H}_{h}u(x_{i},0) + \dfrac{1}{\Gamma(\alpha)}\int_0^{t_{n}} (t_{n}-s)^{\alpha -1}\left(\mathcal{H}_{h}\frac{\partial^2 u}{\partial x^2}(x_{i},s) +  \mathcal{H}_{h}f(x_{i},s)\right)\,ds.
\end{align}
\noindent
We present the discretization in the space variable for \eqref{eq:VE}, which was used in \cite{Zhang2014}.
\begin{lemma*}[\bf{4.1 of \cite{Zhang2014}}]
Let $g(x)$ and $\xi(s)$ be functions such that $g(x) \in C^{6}[x_{i-1}, x_{i+1}]$ and $\xi(s) = 5(1-s)^{3}-3(1-s)^{5}$, then  \begin{align*}
\dfrac{g''(x_{i+1})+10g''(x_{i})+g''(x_{i-1})}{12} &= \dfrac{g(x_{i+1})-2g(x_{i})+g(x_{i-1})}{h^{2}}\\
 &+ \dfrac{h^{4}}{360}\int_0^1 [g^{(6)}(x_{i}-sh)+g^{(6)}(x_{i}+sh)]\xi(s)\, ds.
\end{align*}
\end{lemma*}
\noindent
\subsection{A $C^{1}[0,T]$ in Time Scheme}
The following is an analogue of Lemma 2.1 of \cite{Zhang2014}.
\begin{theorem}
For $0<\alpha<1$ and for $g(t) \in C^1[0,T]$, it follows that
\begin{equation}
\int_0^{t_n}g(s)(t_n -s)^{\alpha-1}\,ds = \sum_{k=1}^{n} \dfrac{g(t_{k-1})+g(t_k)}{2}\int_{t_{k-1}}^{t_k} (t_n-s)^{\alpha-1}\,ds + R_t^n,
\end{equation}
where
\begin{equation*}
|R_t^n|\leq \left(\tau_n+\tau_{max}\right)\frac{T^{\alpha}}{2\alpha} \max_{0\leq t\leq t_n}|g'(t)|, .
\end{equation*}
\begin{proof}
We begin by writing the integral as
\begin{equation*}
\int_0^{t_n}g(s)(t_n -s)^{\alpha-1}\,ds = \int_0^{t_{n-1}}g(s)(t_n -s)^{\alpha-1}\,ds + \int_{t_{n-1}}^{t_n}g(s)(t_n -s)^{\alpha-1}\,ds.
\end{equation*}
The first integral on the right hand side is rewritten as
\begin{align*}
\int_0^{t_{n-1}}g(s)(t_n -s)^{\alpha-1}\,ds &= \sum_{k=1}^{n-1} \int_{t_{k-1}}^{t_k} g(s)(t_n -s)^{\alpha-1}\,ds\\
&= \sum_{k=1}^{n-1} \int_{t_{k-1}}^{t_k}\left(g(s)-\dfrac{g(t_{k-1})+g(t_k)}{2}\right)(t_n -s)^{\alpha-1}\,ds\\
&+ \int_{t_{k-1}}^{t_k}\left(\dfrac{g(t_{k-1})+g(t_k)}{2}\right)(t_n -s)^{\alpha-1}\,ds\\
&= \sum_{k=1}^{n-1}\dfrac{g(t_{k-1})+g(t_k)}{2}\int_{t_{k-1}}^{t_k}(t_n -s)^{\alpha-1}\,ds + (R_1)^n,
\end{align*}
where
\begin{equation*}
(R_1)^n =  \sum_{k=1}^{n-1} \int_{t_{k-1}}^{t_k}\left(g(s)-\dfrac{g(t_{k-1})+g(t_k)}{2}\right)(t_n -s)^{\alpha-1}\,ds.
\end{equation*}
By utilizing the Taylor expansion of $g(s)$ for $s\in (0, t_{n-1})$,
\begin{align}\label{MidpointError1}
|(R_1)^n| &\leq \max_{0\leq t\leq t_{n-1}}|g'(t)|\sum_{k=1}^{n-1} \int_{t_{k-1}}^{t_k}\left|t_k-s-\dfrac{\tau_k}{2} \right|(t_n -s)^{\alpha-1}\,ds \nonumber \\ \nonumber
&\leq \frac{\tau_{max}}{2}\max_{0\leq t\leq t_{n-1}}|g'(t)|\int_{0}^{t_{n-1}}(t_n -s)^{\alpha-1}\,ds\\\	\nonumber
&= \frac{\tau_{max}}{2}\max_{0\leq t\leq t_{n-1}}|g'(t)|\left(\frac{t_n^{\alpha}}{\alpha}-\frac{\tau_n^{\alpha}}{\alpha}\right)\\
&\leq \frac{\tau_{max}T^{\alpha}}{2\alpha}\max_{0\leq t\leq t_{n-1}}|g'(t)| .
\end{align}
In a similar manner, by the Taylor expansion of g(s) for $s\in (t_{n-1},t_n)$, we have
\begin{equation*}
\left|g(s)-\frac{g(t_{n-1})+g(t_n)}{2} \right| \leq \frac{\tau_n}{2}\max_{t_{n-1}\leq t\leq t_n}|g'(t)|, \ \ t_{n-1}< s <t_n.
\end{equation*}
Therefore, the approximation error in the interval $[t_{n-1},t_n]$ satisfies
\begin{align}\label{MidpointError2}
|(R_2)^n| &= \left|\int_{t_{n-1}}^{t_n}\left(g(s)-\dfrac{g(t_{k-1})+g(t_k)}{2}\right)(t_n -s)^{\alpha-1}\,ds\right|\nonumber \\
&\leq \frac{\tau_n^{\iffalse+\alpha\fi}T^{\alpha}}{2\alpha}\max_{t_{n-1}\leq t\leq t_{n}}|g'(t)|.
\end{align}
Finally, since 
$$R_t^n = (R_1)^n + (R_2)^n = \sum_{k=1}^{n} \int_{t_{k-1}}^{t_k}\left(g(s)-\dfrac{g(t_{k-1})+g(t_k)}{2}\right)(t_n -s)^{\alpha-1}\,ds,$$
by combining the error estimates \eqref{MidpointError1} and \eqref{MidpointError2}, we have the desired result.
\end{proof}
\end{theorem}
We remark that under these assumptions, we may obtain a first order accurate scheme for $g(t) \in C^1[0,T]$. The L1-method requires the function $g(t) \in C^2[0,T]$ based on a Taylor series argument, so the condition for the L1-method cannot be relaxed to allow $g(t) \in C^1[0,T]$ due to the nature of the Caputo Fractional Derivative. In section 5, we will see that this scheme exhibits superconvergence for this $C^1[0,T]$ scheme.
\noindent Define
\begin{align}\label{ank}
a^{n}_{k} &= \frac{1}{\Gamma(\alpha)} \int_{t_{k-1}}^{t_{k}}(t_{n}-s)^{\alpha-1}ds, \\
&= \frac{1}{\Gamma(\alpha +1)}\left[ (t_n-t_{k-1})^{\alpha}-(t_n-t_{k})^{\alpha} \right].\nonumber
\end{align}
We define 
\begin{equation}
f_{1-\alpha}(x,t) = \int_0^t \dfrac{(t-s)^{\alpha-1}}{\Gamma(\alpha)}f(x,s)\, ds
\end{equation}
to help succinctly denote the forcing function term in the approximate equation. By applying the $\mathcal{H}_h$ operator, Lemma 4.1 of \cite{Zhang2014}, Lemma 2.2, and the previously stated discretization to \eqref{eq:VE}, we have the fully discretized approximate equation for $u^n_i \approx u(x_i,t_n)$ 
\begin{align}\label{S1}
\mathcal{H}_{h}u_i^n &= \mathcal{H}_{h}\phi(x_i)+\mathcal{H}_h f_{1-\alpha}(x_i,t_n) \nonumber \\ 
&+ \sum_{k=1}^{n}\frac{a_k^n}{2h^2}\left[u^k_{i+1}-2u^k_i+u^k_{i-1}\right]+\frac{a_{k}^n}{2h^2}\left[u^{k-1}_{i+1}-2u^{k-1}_i+u^{k-1}_{i-1}\right],
\end{align}
which is to be solved for $\displaystyle{\{u^n_i\}_{{n=0,1,...,N},\ {i=0,1,...,M}}}$.
\subsection{A $C^2[0,T]$ in Time Scheme}
We begin our findings in this section by establishing a second order method in time, with the restriction of $g(t)\in C^2[0,T]$.
\begin{theorem}\label{C^2thm}
For $0<\alpha<1$ and for $g(t) \in C^2[0,T]$, it follows that
\begin{align}
&\int_0^{t_n}g(s)(t_n -s)^{\alpha-1}\,ds \nonumber\\
&= \sum_{k=1}^{n}\int_{t_{k-1}}^{t_k}\left(\left(1-\dfrac{t_k-s}{\tau_k}\right)g(t_k) + \left(\dfrac{t_k-s}{\tau_k}\right)g(t_{k-1})\right) (t_n-s)^{\alpha-1}\,ds + R_t^n,\label{C2scheme}
\end{align}
where
\begin{equation*}
|R_t^n| \leq \dfrac{\left(\tau_{\max}^2+ \tau_n^2\right)T^{\alpha}}{8\alpha} \max_{0\leq t\leq t_{n-1}}|g''(t)|.
\end{equation*}
\begin{proof}
We begin with the Taylor expansions of $g(s)$ at the points $s=t_k$, where $s\in [t_{k-1},t_k]$, $t_k \in [0,t_n]$ for each $k = 0,1,2,..., N$,
\begin{align*}
g(s) &= g(t_k) + (s-t_k)g'(t_k) + \dfrac{(s-t_k)^2}{2}g''(t_k) + O((s-t_k)^3)\\
g(t_{k-1}) &= g(t_k) - \tau_k g'(t_k) + \dfrac{\tau_k^2}{2}g''(t_k) + O(\tau_k^3).
\end{align*}
Hence, we may combine the above in the following manner:
\begin{align*}
g(s) - &\left(\left(1-\dfrac{t_k-s}{\tau_k}\right)g(t_k) + \left(\dfrac{t_k-s}{\tau_k}\right)g(t_{k-1})\right) \\
&= \left( \dfrac{(s-t_k)^2-\tau_k(t_k-s)}{2}\right)g''(t_k) + O((t_k-s)^3)
\end{align*}
We now rewrite the integral as
\begin{equation*}
\int_0^{t_n}g(s)(t_n -s)^{\alpha-1}\,ds = \int_0^{t_{n-1}}g(s)(t_n -s)^{\alpha-1}\,ds + \int_{t_{n-1}}^{t_n}g(s)(t_n -s)^{\alpha-1}\,ds.
\end{equation*}
The first integral on the right hand side is rewritten as
\begin{align*}
&\int_0^{t_{n-1}}g(s)(t_n -s)^{\alpha-1}\,ds = \sum_{k=1}^{n-1} \int_{t_{k-1}}^{t_k} g(s)(t_n -s)^{\alpha-1}\,ds\\
&= \sum_{k=1}^{n-1} \int_{t_{k-1}}^{t_k}\left(g(s) -\left(\left(1-\dfrac{t_k-s}{\tau_k}\right)g(t_k) + \left(\dfrac{t_k-s}{\tau_k}\right)g(t_{k-1})\right)\right)(t_n -s)^{\alpha-1}\,ds\\
&+ \int_{t_{k-1}}^{t_k}\left(\left(1-\dfrac{t_k-s}{\tau_k}\right)g(t_k) + \left(\dfrac{t_k-s}{\tau_k}\right)g(t_{k-1})\right)(t_n -s)^{\alpha-1}\,ds\\
&= \sum_{k=1}^{n-1}\int_{t_{k-1}}^{t_k}\left(\left(1-\dfrac{t_k-s}{\tau_k}\right)g(t_k) + \left(\dfrac{t_k-s}{\tau_k}\right)g(t_{k-1})\right)(t_n -s)^{\alpha-1}\,ds+ (R_1)^n,
\end{align*}
where then
\begin{equation*}
(R_1)^n =  \sum_{k=1}^{n-1} \int_{t_{k-1}}^{t_k}\left(g(s) - \left(\left(1-\dfrac{t_k-s}{\tau_k}\right)g(t_k) + \left(\dfrac{t_k-s}{\tau_k}\right)g(t_{k-1})\right)\right)(t_n -s)^{\alpha-1}\,ds.
\end{equation*}
We remark that since $s\in [t_{k-1},t_k]$ for each k, then we have $$\left|(s-t_k)^2 - \tau_k (s-t_k)\right|=\left| (s-t_k)(s-t_k-\tau_k) \right|=\left| (s-t_k)(s-t_{k-1}) \right|\leq \frac{\tau_k^2}{4}$$ for each k. By utilizing the Taylor expansion of $g(s)$ for $s\in (0, t_{n-1})$, and by neglecting the higher order terms,
\begin{align}\label{C2Error1}
|(R_1)^n| &\leq \max_{0\leq t\leq t_{n-1}}|g''(t)|\sum_{k=1}^{n-1} \int_{t_{k-1}}^{t_k}\left| \dfrac{(s-t_k)^2-\tau_k(s-t_k)}{2}\right|\left|t_n -s\right|^{\alpha-1}\,ds \nonumber \\ \nonumber
&\leq \dfrac{\tau_{\max}^2}{8} \max_{0\leq t\leq t_{n-1}}|g''(t)|\int_{0}^{t_{n-1}} (t_n-s)^{\alpha-1} \,ds\\\	\nonumber
&\leq \dfrac{\tau_{\max}^2}{8} \max_{0\leq t\leq t_{n-1}}|g''(t)|\left(\frac{t_n^{\alpha}}{\alpha} - \frac{\tau_n^{\alpha}}{\alpha} \right) \\	
&\leq \dfrac{\tau_{\max}^2T^{\alpha}}{8\alpha} \max_{0\leq t\leq t_{n-1}}|g''(t)|.
\end{align}
For a uniform mesh, $\tau_{max} = \tau_n = \tau$, we have the result $|(R_1)^n| \leq C \tau^{2}$.
For the remaining integral term from $[t_{n-1},t_n]$, the same argument is used as for the interval $[0,t_{n-1}]$. Therefore, the approximation error in the interval $[t_{n-1},t_n]$ satisfies
\begin{align}\label{C2Error2}
|(R_2)^n| &= \bigg|\int_{t_{n-1}}^{t_n}\left(g(s) - \left(\left(1-\dfrac{t_k-s}{\tau_k}\right)g(t_k) + \left(\dfrac{t_k-s}{\tau_k}\right)g(t_{k-1})\right)\right)\nonumber \\
&\times (t_n -s)^{\alpha-1}\,ds\bigg|\nonumber \\
&\leq \dfrac{\tau_n^2}{8} \max_{t_{n-1}\leq t\leq t_{n}}|g''(t)| \int_{t_{n-1}}^{t_{n}} (t_n-s)^{\alpha-1} \,ds\nonumber\\
&\leq \dfrac{\tau_n^2}{8} \max_{t_{n-1}\leq t\leq t_{n}}|g''(t)| \left(\frac{\tau_n^a}{\alpha} - 0\right)\nonumber\\
&\leq \dfrac{\tau_n^{2+\alpha}}{8\alpha} \max_{t_{n-1}\leq t\leq t_{n}}|g''(t)|.
\end{align}
Finally, since 
$$R_t^n = (R_1)^n + (R_2)^n = \sum_{k=1}^{n} \int_{t_{k-1}}^{t_k}\left(g(s)-\dfrac{g(t_{k-1})+g(t_k)}{2}\right)(t_n -s)^{\alpha-1}\,ds,$$
by combining the error estimates \eqref{C2Error1} and \eqref{C2Error2}, we have the desired result.
\end{proof}
\end{theorem}
We then arrive at a fully discretized equation for $u^n_i \approx u(x_i,t_n)$ by recalling the definition of $a_k^n$ from \eqref{ank} and by setting 
\begin{align*}
b^n_{1,k} &= \dfrac{1}{\Gamma(\alpha)}\int_{t_{k-1}}^{t_k} \left(1-\dfrac{t_k-s}{\tau_k}\right)(t_n-s)^{\alpha-1}\,ds\\ 
b^n_{2,k} &=  \dfrac{1}{\Gamma(\alpha)}\int_{t_{k-1}}^{t_k} \left(\dfrac{t_k-s}{\tau_k}\right)(t_n-s)^{\alpha-1}\,ds
\end{align*}
The resulting fully discretized equation is as follows:
\begin{align}\label{S2}
\mathcal{H}_{h}u_i^n &= \mathcal{H}_{h}\phi(x_i)+\mathcal{H}_h f_{1-\alpha}(x_i,t_n) \nonumber \\ 
&+ \sum_{k=1}^{n}\left(b_{1,k}^n\left[\dfrac{u^k_{i+1}-2u^k_i+u^k_{i-1}}{h^2}\right] +b_{2,k}^n\left[\dfrac{u^{k-1}_{i+1}-2u^{k-1}_i+u^{k-1}_{i-1}}{h^2} \right]\right),
\end{align}
which is to be solved for $\displaystyle{\{u^n_i\}_{{n=0,1,...,N},\ {i=0,1,...,M}}}$.
\section{Error Estimates}
Before we establish stability and convergence of the numerical methods used, we will make use of the definitions in [1, p.202].  Let $$\mathcal{V}_{h} = \{ v=(v_{0},v_{1},..., v_{M})\vert v_{0}=v_{M}=0\}.$$ For any grid functions $v,w\in \mathcal{V}_{h}$, we will define the following:
\begin{align*}
L_{2} \text{ norm} &\norm{v}_{h} = \sqrt{<v,v>_{h}} \\
H^{1} \text{ semi-norm} &\norm{\delta_{x}v}_{h} = \sqrt{h\sum_{i=1}^{M}(\delta_{x}v_{i-1})^{2}}\\
H^{1} \text{ norm} &\norm{v}_{1,h} = \sqrt{\norm{v}_h^2+\norm{\delta_x v}_h^2}
\end{align*}
Where $\norm{\mathcal{H}_{h}v}_{h}$ and $\norm{\delta_{x}^{2}v}_{h}$ are defined in a similar manner. By applying Lemma 4.2 of \cite{Zhang2014}, then 
\begin{equation*}
\norm{v}_{h} \leq \dfrac{1}{\sqrt{6}}\norm{\delta_{x}v}_{h}.
\end{equation*}
Following \cite{Zhang2014}, define
\begin{equation*}
<v,w>_{A} = h\sum_{i=1}^{M}(\delta_{x}v_{i-1/2}\boldsymbol{\cdot}\delta_{x}w_{i-1/2})-\dfrac{h^{2}}{12}h\sum_{i=1}^{M-1}\delta_{x}^{2}v_{i}\boldsymbol{\cdot}\delta_{x}^{2}w_{i},
\end{equation*}
and
\begin{equation*}
\norm{v}_{A} = \sqrt{<v,v>}_{A}.
\end{equation*}
They further go on to show that, by Lemma 4.3 of \cite{Zhang2014},
\begin{equation*}
-h\sum_{i=1}^{M-1}(\mathcal{H}_{h}v_{i})\boldsymbol{\cdot}\delta_{x}^{2}w_{i} = <v,w>_{A},
\end{equation*}
which establishes that $\norm{\boldsymbol{\cdot}}_{A}$ and $\norm{\delta_{x}\boldsymbol{\cdot}}_{h}$ are equivalent. 
\subsection{Consistency, Stability, and Convergence Results}
With the preliminaries established in section 2, we will present the main theorems of this paper. We begin with deriving the consistency of the schemes \eqref{S1} and \eqref{S2} and then the stability for each. With both these proofs, we are able to assert the convergence of each scheme, which is then demonstrated in the next section.
\begin{theorem}
Let $\{u_{i}^{n}| 0\leq i \leq M, 1 \leq n \leq N\}$ be the solution of the approximate scheme \eqref{S1}, with a uniform grid used in the spatial domain. Further, let $\phi,f(\boldsymbol{\cdot},t),f_{t}(\boldsymbol{\cdot},t),f_{tt}(\boldsymbol{\cdot},t) \in D(\mathcal{L}^{9/2})$ for each $t \in (0,T]$. Then, u is a unique solution to \eqref{eq:VE}, with resulting approximation error
\begin{equation}
\norm{u(x_{i},t_{j})-u_{i}^{n}}_A \leq \dfrac{T^{\alpha}}{\Gamma(\alpha+1)}\left(\dfrac{h^4}{180}\norm{\dfrac{\partial ^{6}u}{\partial x^{6}}}_{\infty}+\left(\dfrac{\tau_n+\tau_{max}}{2}\right) \norm{\dfrac{\partial u}{\partial t}}_{\infty}\right).
\end{equation} 
\end{theorem}
\begin{proof}
By Theorem A, there exists a unique solution to \eqref{eq:VE}. Denote the residual of the approximation $(R_{x})^{n}(x_{i},t_n)$ by 
\begin{equation} \label{Residual}
(R_{x})^{n}(x_{i},t_n) = \dfrac{h^{4}}{360}\int_0^1 \left[\dfrac{\partial^6 u}{\partial x^6}(x_{i}-sh,t_n)+\dfrac{\partial^6 u}{\partial x^6}(x_{i}+sh,t_n)\right]ds
\end{equation}
for all $t\in [0,1]$. We may then bound $(R_{x})^{n}(x_{i},t_n)$ by
\begin{align*}
\bigg|(R_{x})^{n}(x_{i},t_{n}) \bigg|&= \bigg|\dfrac{h^{4}}{360}\int_0^1 \left(\dfrac{\partial^{6}u}{\partial x^{6}}(x_{i}-sh,t_n)+\dfrac{\partial ^{6}u}{\partial x^{6}}(x_{i}+sh,t_n)\right)\,ds \bigg|\\
\bigg|(R_{x})^{n}(x_{i},t_{j})\bigg|&\leq \dfrac{h^{4}}{360}\int_0^1 \left( \norm{\dfrac{\partial^{6}u}{\partial x^{6}}}_{\infty}+\norm{\dfrac{\partial ^{6}u}{\partial x^{6}}}_{\infty}\right)\,ds\\
&= \dfrac{h^{4}}{180}\norm{\dfrac{\partial ^{6}u}{\partial x^{6}}}_{\infty}.
\end{align*}
By applying Lemma 4.1 of \cite{Zhang2014} to \eqref{DifferenceEQ}, we see that 
\begin{align*}
 &\mathcal{H}_{h}u(x_{i},t_{n}) = \mathcal{H}_{h}u(x_{i},0) + \dfrac{1}{\Gamma(\alpha)}\int_0^{t_{n}} (t_{j}-s)^{\alpha -1}(\mathcal{H}_{h}u_{xx}(x_{i},s) +  \mathcal{H}_{h}f(x_{i},s))\,ds\\
 &= \dfrac{1}{\Gamma(\alpha)}\int_0^{t_{n}} \bigg(\dfrac{u(x_{i+1},s)-2u(x_{i},s)+u(x_{i-1},s)+h^{2}(R_{x})^{n}(x_{i},s)}{(t_{j}-s)^{1-\alpha}h^{2}}\bigg)\,ds\\ 
 &+ \mathcal{H}_{h}f_{1-\alpha}(x_{i},t_{j})).
\end{align*}
Further, by applying Theorem 3.1 to \eqref{DifferenceEQ}, we have
\begin{align*}
\mathcal{H}_{h}u_i^n &= \mathcal{H}_{h}\phi(x_i)+\mathcal{H}_h f_{1-\alpha}(x_i,t_n) \nonumber \\ 
&+ \sum_{k=1}^{n}\left(\frac{a_k^n}{2h^2}\left[u^k_{i+1}-2u^k_i+u^k_{i-1}\right]+\frac{a_{k}^n}{2h^2}\left[u^{k-1}_{i+1}-2u^{k-1}_i+u^{k-1}_{i-1}\right]\right)\\
&+\dfrac{1}{\Gamma(\alpha)}\int_0^{t_{n}}\left((R_{x})^{n}(x_{i},s)+(R_t)^n(x_{i},s)\right)(t_{n}-s)^{\alpha-1}\,ds.
\end{align*}
Finally, we see that the approximation error is 
\begin{align*}
\norm{u(x_{i},t_{n})-u_i^n}_A &= \norm{\dfrac{1}{\Gamma(\alpha)}\int_0^{t_{n}}\left((R_{x})^{n}(x_{i},s)+(R_t)^n(x_{i},s)\right)(t_{n}-s)^{\alpha-1}\,ds}_A\\
&\leq \dfrac{T^\alpha}{\Gamma(\alpha+1)}(\norm{(R_{x})^{n}}_A+\norm{(R_t)^n}_A)\\
&\leq \dfrac{T^{\alpha}}{\Gamma(\alpha+1)}\left(\dfrac{h^4}{180}\norm{\dfrac{\partial ^{6}u}{\partial x^{6}}}_{\infty}+\left(\dfrac{\tau_n+\tau_{max}}{2}\right) \norm{\dfrac{\partial u}{\partial t}}_{\infty}\right).
\end{align*}
\end{proof}

We will remark that as $\alpha \rightarrow 0$ then $\dfrac{T^{\alpha}}{\Gamma(1+\alpha)} \rightarrow 1$. Also, as $\alpha \rightarrow 1$ then $\dfrac{T^{\alpha}}{\Gamma(1+\alpha)} \rightarrow T$. The following corollary is immediate from the previous theorem. 
\begin{corollary}
Under a uniform partition of the time domain where $\tau_n = \tau$ for all $n$, then the approximation error of \eqref{S1} is $O(h^4+\tau)$. 
\end{corollary}
We also have a theorem asserting the stability of the discrete scheme and derives the corresponding error equations of the scheme:
\begin{theorem}
Suppose $\{u_i^n \vert 0\leq i\leq M, 1\leq n \leq N\}$ is the solution of the difference scheme \eqref{S1}. Then, for any size temporal mesh described before, the discrete difference scheme \eqref{S1} is unconditionally stable to f and $\phi$, where
\begin{equation*}
\norm{u^{n}}_{A}^{2} \leq \norm{\phi}_{A}^{2}+\frac{T^{\alpha}}{\Gamma(\alpha+1)}\max_{1\leq l\leq N}\norm{\mathcal{H}_{h}f^l}_{h}^{2}
\end{equation*} 
\end{theorem}
\begin{proof} 
Recall that
\begin{align*}
a^{n}_{k} &= \frac{1}{\Gamma(\alpha)} \int_{t_{k-1}}^{t_{k}}(t_{n}-s)^{\alpha-1}ds,\\
&= \frac{1}{\Gamma(\alpha +1)}\left[(t_n-t_{k-1})^{\alpha}- (t_n-t_k)^{\alpha} \right].
\end{align*}
We consider the scheme \eqref{S1} after combining the initial and boundary conditions. By omitting the residual term $R_i^n$ and by substituting the exact solution $U^k_i$ with its approximation $u^k_i$ into \eqref{S1}, we have:
\begin{equation*}
\mathcal{H}_h u^n_i = \mathcal{H}_h u^0_i + \sum_{k=1}^{n} a^n_k\left( \frac{\delta_x^2 u^k_i+\delta_x^2 u^{k-1}_i}{2}+ \mathcal{H}_h  f^n_i\right) .
\end{equation*}
By multiplying both sides by $-2h\delta_x^2u^n_i$ and summing over each i, then
\begin{align*}
2\norm{u^n}_A^2 + \sum_{k=1}^{n-1} a^n_k&\left(\norm{\delta_x^2 u^k}_h^2 + \norm{\delta_x^2 u^{k-1}}_h^2 \right) = 2<u^0, u^n>_A - 2 \sum_{k=1}^{n}a^n_k<\mathcal{H}_h f, \delta_x^2 u^n>_h\\
&\leq \left( \norm{u^0}_A^2+\norm{u^n}_A^2\right)+\sum_{k=1}^{n}a^n_k\left(\norm{\mathcal{H}_h f^n}_h^2 +\norm{\delta_x^2 u^n}_h\right)\\
\Rightarrow \norm{u^n}_A^2 &\leq \norm{\phi}^2_A + \sum_{k=1}^{n} a^n_k \max_{1\leq l \leq N}\norm{\mathcal{H}_h f^l}_h^2 \ \ \  1\leq n\leq N.
\end{align*}
Finally, since $ \sum_{k=1}^{n}a^n_k = \frac{T^{\alpha}}{\Gamma(\alpha+1)}$ when $n=N$, we see the result holds.
\end{proof}
To further see the convergence of the numerical scheme, denote $\epsilon^n_i:= u(x_i,t_n)-u_i^n$. The error equations are then obtained:
\begin{equation} \label{Erroreq}
\mathcal{H}_h \epsilon ^n_i = \sum_{k=1}^n a^n_k \delta_x^2 \epsilon ^n_i + a^n_n R^n_i 
\end{equation}
\begin{equation*}
\epsilon ^n_0 = \epsilon ^n_M = 0 \text{,      } 1\leq n\leq N
\end{equation*}
\begin{equation*}
\epsilon ^0_i =0 \text{,      } 0\leq i\leq M.
\end{equation*}
By applying \eqref{Residual} and by applying the previous stability analysis, we have the immediate error convergence result
\begin{align*}
\norm{\epsilon ^n}_A^2 &\leq \norm{\epsilon ^0}_A^2 + \frac{T^\alpha}{\Gamma(\alpha+1)}\norm{R^n_i}_h^2\\
&\leq \frac{T^\alpha}{\Gamma(\alpha+1)}\left( \dfrac{h^4}{180}\norm{\dfrac{\partial ^{6}u}{\partial x^{6}}}_{\infty}+ \left( \dfrac{\tau_n+\tau_{max}}{2} \right)  \norm{\dfrac{\partial u}{\partial t}}_{\infty}\right)^2,\\
\norm{\epsilon ^n}_A &\leq \sqrt{\dfrac{T^{\alpha}}{\Gamma(\alpha+1)}}\left( \dfrac{h^4}{180}\norm{\dfrac{\partial ^{6}u}{\partial x^{6}}}_{\infty}+ \left( \dfrac{\tau_n+\tau_{max}}{2} \right) \norm{\dfrac{\partial u}{\partial t}}_{\infty}\right).
\end{align*}
That is, the scheme \eqref{S1} is both stable and consistent, hence it is convergent, see [7-9] for further details. Therefore, by [3, theorem 2.1] we have the following immediate results:
\begin{theorem}
Let $\{u_{i}^{n}| 0\leq i \leq M, 1 \leq n \leq N\}$ be the solution of the approximate scheme \eqref{S1}, with a uniform grid used in the spatial domain and any grid spacing used in the temporal direction. Further, let $\phi,f(\boldsymbol{\cdot},t),f_{t}(\boldsymbol{\cdot},t),f_{tt}(\boldsymbol{\cdot},t) \in D(\mathcal{L}^{9/2})$ for each $t \in (0,T]$. Then, it holds for some $C>0$
\begin{equation}
\norm{u(x_i,t_n)-u_i^n}_A \leq \sqrt{\dfrac{T^{\alpha}}{\Gamma(\alpha+1)}}C\left(h^4+\tau_{max} \right), \ \ \ \ 1\leq n \leq N.
\end{equation}
\end{theorem}
We also have a corollary detailing the use of a truncation of the exact solution to generate the data at $u(x,t_1)$.
\begin{corollary}\label{truncation}
Let $u_{h,1}(x,t_1) = \phi(x) + \dfrac{\phi''(x)t_1^{\alpha}}{\Gamma(\alpha+1)}+ (f*a_{1-\alpha})(x,t_1)$. Then, the truncation error 
$$\norm{e_{h,1}}_{\infty} = \norm{u(x,t_1)-u_{h,1}(x,t_1)}_{\infty} \leq C\tau_{max}^{2\alpha}\left(\norm{\phi^{(4)}(x)}_{\infty}+\norm{f(x,t_1)}_{\infty} \right)$$
\end{corollary}
\begin{proof}
Consider the exact solution $u(x,t)$ of \eqref{eq:VE} which is generated from \eqref{unique_soln}.  That is, 
\begin{align*}u(x,t) &= \phi(x) + \dfrac{\phi''(x)t^{\alpha}}{\Gamma(\alpha+1)}+\dfrac{\phi^{(4)}(x)t^{2\alpha}\Gamma(1/2)}{4^{\alpha}\Gamma(\alpha+1/2)}+O(t^{3\alpha}\phi^{(6)}(x))+(f*a_{1-\alpha})(x,t) \\
&+ ((f*a_{1-\alpha})*a_{1-\alpha})(x,t) + O(t^{3\alpha}).
\end{align*}
Hence, the truncation error at the first time step $t_1$ is, after ignoring the higher order terms, 
\begin{align*}
\norm{e_{h,1}}_{\infty} &= \norm{\dfrac{\phi^{(4)}(x)t_1^{2\alpha}\Gamma(1/2)}{4^{\alpha}\Gamma(\alpha+1/2)}+((f*a_{1-\alpha})*a_{1-\alpha})(x,t_1)}_{\infty}\\
&\leq Ct_1^{2\alpha}\norm{\phi(x)}_{\infty}+\norm{\int_0^{t_1}\dfrac{(t_1-s)^{\alpha-1}}{\Gamma(\alpha)}\left(\int_0^s \dfrac{(s-v)^{\alpha-1}f(x,v)}{\Gamma(\alpha)} \,dv\right)\, ds}\\
&\leq Ct_1^{2\alpha}\norm{\phi(x)}_{\infty} + \norm{f(x,t_1)}_{\infty}\norm{\int_0^{t_1}\dfrac{s^{\alpha}(t_1-s)^{\alpha-1}}{\Gamma(\alpha)}}_{\infty}\\
&= Ct_1^{2\alpha}\norm{\phi(x)}_{\infty} + \norm{f(x,t_1)}_{\infty}\norm{\dfrac{\Gamma(1/2)t_1^{2\alpha}}{4^{\alpha}\Gamma(\alpha+1/2)}}_{\infty}\\
&\leq C t_1^{2\alpha}\left(\norm{\phi(x)}_{\infty}+\norm{f(x,t_1)}_{\infty}\right).
\end{align*}
\end{proof}
\begin{remark*}
By letting $t_1 = \tau,\, \phi(x) = 0,\,$ and where $f(x,\tau) = (\tau+O(\tau^{1+\alpha}))X(x)$, we have the truncation error in \cref{truncation} after neglecting the terms of order $O(\tau^{1+\alpha})$ :
$$\norm{e_{h,1}}_{\infty} = \norm{u(x,t_1)-u_{h,1}(x,t_1)}_{\infty} \leq C\tau^{1+2\alpha}\norm{X(x)}_{\infty}. $$
\end{remark*}
We have a similar set of results for the numerical scheme \eqref{S2} for functions $g(t) \in C^2[0,T]$. Beginning with the consistency results, we will provide each theorem as follows:
\begin{theorem}
Let $\{u_{i}^{n}| 0\leq i \leq M, 1 \leq n \leq N\}$ be the solution of the approximate scheme \eqref{S2}, with a uniform grid used in the spatial domain. Further, let $\phi,f(\boldsymbol{\cdot},t),f_{t}(\boldsymbol{\cdot},t),f_{tt}(\boldsymbol{\cdot},t) \in D(\mathcal{L}^{9/2})$ for each $t \in (0,T]$. Then, u is a unique solution to \eqref{eq:VE}, with resulting approximation error
\begin{equation}
\norm{u(x_{i},t_{j})-u_{i}^{n}}_A \leq \dfrac{T^{\alpha}}{\Gamma(\alpha+1)}\left(\dfrac{h^4}{180}\norm{\dfrac{\partial ^{6}u}{\partial x^{6}}}_{\infty}+\left(\dfrac{\tau_n^2+\tau_{max}^2}{8}\right) \norm{\dfrac{\partial^2 u}{\partial t^2}}_{\infty}\right).
\end{equation}
\begin{proof} 
The proof is identical to Theorem 4.1 and is omitted. 
\end{proof}
\end{theorem}
Therefore, under a uniform partition of the time domain, the approximation error of \eqref{S2} is $O(h^4+\tau^2)$. 
\begin{theorem}
Suppose $\{u_i^n \vert 0\leq i\leq M, 1\leq n \leq N\}$ is the solution of the difference scheme \eqref{S2}. Then, for any size temporal mesh described before, the discrete difference scheme \eqref{S2} is unconditionally stable to f and $\phi$, where
\begin{equation*}
\norm{u^{n}}_{A}^{2} \leq \norm{\phi}_{A}^{2}+\frac{T^{\alpha}}{\Gamma(\alpha+1)}\max_{1\leq l\leq N}\norm{\mathcal{H}_{h}f^l}_{h}^{2}
\end{equation*} 
\end{theorem}
\begin{proof}
The proof is identical to Theorem 4.3 and is omitted. 
\end{proof}
Following our stability result, the error equations are then obtained:
\begin{equation} \label{Erroreq}
\mathcal{H}_h \epsilon ^n_i =\sum_{k=1}^n a^n_k \delta_x^2 \epsilon ^n_i + R^n_i
\end{equation}
\begin{equation*}
\epsilon ^n_0 = \epsilon ^n_M = 0 \text{,      } 1\leq n\leq N
\end{equation*}
\begin{equation*}
\epsilon ^0_i =0 \text{,      } 0\leq i\leq M.
\end{equation*}
In a similar manner to the convergence result for the scheme \eqref{S1}, we also have an error convergence result for \eqref{S2}
\begin{align*}
\norm{\epsilon ^n}_A^2 &\leq \norm{\epsilon ^0}_A^2 + \frac{T^\alpha}{\Gamma(\alpha+1)}\norm{R^n_i}_h^2\\
&\leq \frac{T^\alpha}{\Gamma(\alpha+1)}\left( \dfrac{h^4}{180}\norm{\dfrac{\partial ^{6}u}{\partial x^{6}}}_{\infty}+ \left( \dfrac{\tau_n^2+\tau_{max}^2}{8} \right)  \norm{\dfrac{\partial^2 u}{\partial t^2}}_{\infty}\right)^2\\
\norm{\epsilon ^n}_A &\leq \sqrt{\dfrac{T^{\alpha}}{\Gamma(\alpha+1)}}\left( \dfrac{h^4}{180}\norm{\dfrac{\partial ^{6}u}{\partial x^{6}}}_{\infty}+ \left( \dfrac{\tau_n^2+\tau_{max}^2}{8} \right) \norm{\dfrac{\partial^2 u}{\partial t^2}}_{\infty}\right).
\end{align*}
Hence, the schemes \eqref{S1} and \eqref{S2} are both stable and consistent, hence they are both convergent. Therefore, by [3, theorem 2.1] we have the following immediate results:
\begin{theorem}
Let $\{u_{i}^{n}| 0\leq i \leq M, 1 \leq n \leq N\}$ be the solution of the approximate scheme \eqref{S2}, with a uniform grid used in the spatial domain and any grid spacing used in the temporal direction. Further, let $\phi,f(\boldsymbol{\cdot},t),f_{t}(\boldsymbol{\cdot},t),f_{tt}(\boldsymbol{\cdot},t) \in D(\mathcal{L}^{9/2})$ for each $t \in (0,T]$. Then, it holds for some $C>0$
\begin{equation}
\norm{u(x_i,t_n)-u_i^n}_A \leq \sqrt{\dfrac{T^{\alpha}}{\Gamma(\alpha+1)}}C\left(h^4+\tau_{max}^2 \right), \ \ \ \ 1\leq n \leq N.
\end{equation}
\end{theorem}
These results imply that under the same regularity assumptions in \cite{Zhang2014}, we may improve our order of convergence by a factor of $\alpha$. Further, we may also relax these regularity assumptions to have $g(t) \in C^1[0,T]$ while preserving an order of convergence of $O(k)$ in the time variable, which is not possible with the L1-method. In the next section we shall consider a simple numerical experiment that illustrates our theoretical results. 
\section{Numerical Experiment}
We will consider the following test problem for our numerical experiments:
\begin{align*}
u(x,t)&=\sin(\pi x)t^2 \text{, } u(0,t)=u(1,t)=0 \text{, } \phi = u(x,0) = 0 \text{, }\\
f_{1-\alpha}(x,t)&= \sin(\pi x)\left[t^2 + \dfrac{2\pi^2t^{\alpha+2}}{\Gamma(\alpha+3)}\right] = a_{1-\alpha}(t)*f(x,t),
\end{align*}
which will satisfy $u(x,t) \in C^{1}[0,T], C^{2}[0,T]$ in time. We will define $M$ to be the number of partitions of the spatial domain, $E_{2}(M,N)$ to be the maximum error attained over the total mesh for a uniform mesh for functions in $C^2[0,T]$, and $\text{rate}_2 = log_{2}\Big( \dfrac{E_{2}(M,N/2)}{E_{2}(M,N)}\Big)$.  Therefore, for $M=25$ and $T=1$, we have the following:
\\
\\
\renewcommand{\arraystretch}{0.35}
\begin{tabular}{ |p{3cm}||p{3cm}|p{3cm}|p{3cm}|  }
 \hline
 \multicolumn{4}{|c|}{Numerical Error for $u(x,t)=\sin(\pi x)t^2$, T=1 on a Uniform mesh} \\
 \hline
 $\alpha$ & N & $E_{1}(M,N)$ & $\text{rate}_{1}$\\
 \hline
  $0.05$   & 10    &$0.0786$&  *\\
 &20 &$0.0370$ &$1.089$ \\
 &40 &$0.0179$ &$1.046$ \\
 &80 &$0.0087$ &$1.046$ \\
 &160 &$0.0042$	&$1.0477$ \\
 \hline
 $0.25$   & 10    &$0.0370$&  *\\
 &20 &$0.0157$ &$1.2388$ \\
 &40 &$0.0067$ &$1.222$ \\
 &80 &$0.0029$ &$1.2329$ \\
 &160 &$0.0012$	&$1.2398$ \\
 \hline
 $0.5$   & 10    &$0.0122$&  *\\
 &20 &$0.0046$ &$1.4884$ \\
 &40 &$0.0017$ &$1.4294$ \\
 &80 &$6.276$e--4 &$1.4517$ \\
 &160 &$2.2705$e--4	&$1.4668$ \\
 \hline
 $0.75$   & 10    &$0.0078$&  *\\
 &20 &$0.0014$ &$2.4316$ \\
 &40 &$2.7567$e--4 &$2.3859$ \\
 &80 &$9.1448$e--5 &$1.5919$ \\
 &160 &$2.9572$e--5	&$1.6287$ \\

 \hline
 $0.95$   & 10    &$0.0046$&  *\\
 &20 &$7.76962$e--4 &$2.6739$ \\
 &40 &$1.1287$e--4 &$2.7832$ \\
 &80 &$1.5854$e--5 &$2.8317$ \\
 &160 &$3.7211$e--6	&$2.0911$ \\

 \hline
\end{tabular}\\
\\
The results for using the $C^2[0,T]$ scheme for the same test problem is as follows:
\\
\renewcommand{\arraystretch}{0.35}
\begin{tabular}{ |p{3cm}||p{3cm}|p{3cm}|p{3cm}|  }
 \hline
 \multicolumn{4}{|c|}{Numerical Error for $u(x,t)=\sin(\pi x)t^2$, using $C^2[0,T]$ scheme} \\
 \hline
 $\alpha$ & N & $E_{2}(M,N)$ & $\text{rate}_{2}$\\
 \hline
  $0.05$   & 10    &$0.0003$&  *\\
 &20 &$9.5503$e--5 &$1.8065$ \\
 &40 &$2.5878$e--5 &$1.8838$ \\
 &80 &$6.5375$e--6 &$1.9849$ \\
 &160 &$1.6291$e--6	&$2.0047$ \\
 \hline
 $0.25$   & 10    &$0.0011$&  *\\
 &20 &$0.0003$ &$1.8854$ \\
 &40 &$7.8046$e--5 &$1.9375$ \\
 &80 &$1.9597$e--5 &$1.9937$ \\
 &160 &$4.8246$e--6	&$2.0222$ \\
 \hline
 $0.5$   & 10    &$0.0014$&  *\\
 &20 &$0.0004$ &$1.9483$ \\
 &40 &$9.4512$e--5 &$1.9847$ \\
 &80 &$2.3281$e--5 &$2.0214$ \\
 &160 &$5.6626$e--6	&$2.0396$ \\
 \hline
 $0.75$   & 10    &$0.0016$&  *\\
 &20 &$0.0004$ &$1.9801$ \\
 &40 &$9.9404$e--5 &$2.0014$ \\
 &80 &$2.4403$e--5 &$2.0262$ \\
 &160 &$5.9182$e--6	&$2.0439$ \\

 \hline
 $0.95$   & 10    &$0.0016$&  *\\
 &20 &$0.0004$e--4 &$1.9974$ \\
 &40 &$0.0001$e--5 &$2.0067$ \\
 &80 &$2.5303$e--5 &$2.0165$ \\
 &160 &$6.1781$e--6	&$2.0341$ \\

 \hline
\end{tabular}\\
\\
The above table show that for various values of $\alpha$, the error estimate improves with an increase in the number of space and time steps used in the mesh partitioning while preserving a rate of convergence of $O(h^4 + k^2)$ as expected. As a result, our method exhibits a better rate of convergence overall, under the same regularity assumptions. By \cref{truncation}, if we instead replace $u(x_i,t_1)$ with its approximation derived from the exact solution, we instead have the following improved results for a small amount of time steps due to the truncation error. For this example, we have $u(x_i,t_1) = f_{1-\alpha}(x_i,t_1)$ These results are summarized in the following table:
\\
\\
\renewcommand{\arraystretch}{0.35}
\begin{tabular}{ |p{3cm}||p{3cm}|p{3cm}|p{3cm}|  }
 \hline
 \multicolumn{4}{|c|}{Numerical Error for $u(x,t)=\sin(\pi x)t^2$ with truncation error for $u(x,t_1)$} \\
 \hline
 $\alpha$ & N & $E_{1}(M,N)$ & $\text{rate}_{1}$\\
 \hline
  $0.05$   & 10    &$0.0840$&  *\\
 &20 &$0.0370$ &$1.1821$ \\
 &40 &$0.0179$ &$1.046$ \\
 \hline
 $0.25$   & 10    &$0.0435$&  *\\
 &20 &$0.0157$ &$1.4744$ \\
 &40 &$0.0067$ &$1.222$ \\
 \hline
 $0.5$   & 10    &$0.0189$&  *\\
 &20 &$0.0046$ &$2.0227$ \\
 &40 &$0.0017$ &$1.4294$ \\
 \hline
 $0.75$   & 10    &$0.0079$&  *\\
 &20 &$0.0014$ &$2.4636$ \\
 &40 &$2.7567$e--4 &$2.3859$ \\
 \hline
 $0.95$   & 10    &$0.0050$&  *\\
 &20 &$7.76962$e--4 &$2.6739$ \\
 &40 &$1.1287$e--4 &$2.7832$ \\
 \hline
\end{tabular}\\
\\
\section{Conclusion}
We have shown that using the Laplace transform on the Caputo fractional derivative can preserve the maximum accuracy of an estimate, with some improvements depending on the value of $\alpha$. By a Taylor Series expansion to approximate the convolution integral, we may assert that one can design a scheme that more accurately approximates this problem for certain values of $\alpha$ over the desired meshes. Further, utilizing this Taylor Series expansion argument, we are able to derive a scheme that only requires the function to be in $C^1(0,T]$ for the time variable, which allows for a wider class of functions. We also present a scheme for $C^2[0,T]$ functions that parallels the L1-method, as seen in [1,3,15,16], which has error of $O(k^2)$ in time. This novel result improves over previous results, which guarantee an error of $O(k^{2-\alpha})$ in time for the same regularity assumption. 

\newpage

\appendix

\section{Existence and Uniqueness of a Solution to \eqref{eq:VE}}

Consider the Hilbert space $L^2(0,1)$ and let $\sigma(A)$ denote the spectrum of the operator $A=-\dfrac{\partial^2}{\partial x^2}$ which is a strictly positive self-adjoint operator on the dense subspace $H_0^2(0,1)$.
The operator valued equation $(\lambda I- A)(X)=0$ has the solution 
\begin{align*}
(\lambda I- A)(X) &= (\lambda X - A(X))=0\\
&= X''+\lambda X =0\\
&\Rightarrow X_{\lambda}(x) = \sin(\sqrt{\lambda} x) \text{, with Eigenvalues } \lambda_{n}=(n\pi)^2 .
\end{align*}
Now, let 
\begin{equation*}
\delta_A = \inf_{{y \neq 0} \text{, }{y \in H_0^2((0,1))}} \frac{(Ay,y)}{(y,y)}=\pi^2.
\end{equation*}
\indent It is easy to see that $\displaystyle{a_{1-\alpha}(t)=\frac{t^{\alpha-1}}{\Gamma(\alpha)}}$ is positive, decreasing on $(0,\infty)$, and   $a_{1-\alpha} \in C(0,\infty) \cap L^{1}(0,1)$. Therefore, we may apply Theorem 4.1 from \cite{Friedman1969} to see that the operator S(t) defined as 
\begin{equation*}
S(t)x_0=\int_{\delta_{A}}^{\infty} S_{\lambda}(t)dE_{\lambda}x_0\ \ \  (x_0 \in L^{2}(0,1)),
\end{equation*}
is the fundamental solution of \eqref{eq:VE}, as defined in \cite{Friedman1969}. Here, $S_{\lambda} = S_{\lambda}(t)$ is the solution of the scalar equation
\begin{equation} \label{Friedman_3.2}
S_{\lambda}(t) = 1-\lambda \int_s^t a_{1-\alpha}(t-\tau)S_{\lambda}(\tau)\,d\tau,
\end{equation}
and $E_{\lambda}$ is the resolution of the identity for $A$ and because the operator-valued function $S=S(t)$ is a fundamental solution, $S \in L^1\left((0,T]; \mathcal{B}\left(L^2(0,1)\right)\right)$, and for almost all $t\in [0,T]$ 
\begin{equation*}
S(t) = I-A\int_0^t a_{1-\alpha}(t-\tau)S(\tau)\,d\tau
\end{equation*}
where $\mathcal{B}\left(L^2(0,1)\right)$ is the space of all bounded linear operators of $L^2(0,1)$ and $I$ is the identity operator. If $\phi(x) = \sum\limits_{n=1}^{\infty}  a_n \sin(n\pi x)\in H^2_0(0,1)$ then $\displaystyle{\sum_{n=1}^{\infty}\vert a_n\vert^2 n^2 <\infty}$, and 
\begin{align*}
A\phi(x) &= \int_{\delta_A}^{\infty}\lambda dE_{\lambda}\phi(x)\\
&= \sum\limits_{n=1}^{\infty} \lambda_n a_n \sin(n\pi x)\\
&= \sum\limits_{n=1}^{\infty} (\pi n)^2 a_n \sin(n \pi x).
\end{align*}
Let $\hat{a}_{1-\alpha}(s) = \mathcal{L}(a_{1-\alpha}(t))$. Define $g(s) = s \hat{a}_{1-\alpha}(s) = s(s^{-\alpha}) = s^{1-\alpha}$. We may calculate $S_{\lambda}$ using the following from \cite{Friedman1969}:
\begin{align*}
S_{\lambda} &= \mathcal{L}^{-1}\bigg(\frac{1}{s+\lambda g(s)}\bigg) = \mathcal{L}^{-1}\bigg(\frac{1}{s+\lambda s^{1-\alpha}}\bigg)\\
&= \mathcal{L}^{-1}\bigg(\frac{s^{-1}}{1+\lambda s^{-\alpha}}\bigg) = E_{\beta}(-\lambda t^{\alpha}),
\end{align*}
where $E_{\beta}$ is the well known Mittag-Leffler function, $\displaystyle E_{\beta}(z) = \sum_{n=0}^{\infty} \frac{z^{n}}{\Gamma(1+n\beta)}$, see Theorem 6.1.1 of \cite{Brunner2004} for more details. Now, from the above calculations, 
\begin{align}\label{unique_soln}
S(t)\phi(x)=\int_{\delta_{a}}^{\infty} S_{\lambda}(t)dE_{\lambda}\phi(&x) = \sum\limits_{n=1}^{\infty} S_{\lambda n}(t) a_n \sin(n\pi x)\nonumber\\
&= \sum\limits_{n=1}^{\infty} \sum\limits_{m=0}^{\infty} \bigg[\frac{(-\lambda_{n} t^{\alpha})^{m}}{\Gamma(m \alpha +1)}\bigg]a_n \sin(n\pi x)\nonumber\\
&= \sum\limits_{n=1}^{\infty} a_n \sin(n\pi x) \sum\limits_{m=0}^{\infty} \bigg[\frac{(-\lambda_{n} t^{\alpha})^{m}}{\Gamma(m \alpha +1)}\bigg].
\end{align}
Then $$u(x,t) = S(t)\phi(x)+(S*f)(x,t)$$ is, in closed form, the unique solution of \eqref{eq:VE}, ensured by Theorem A.

\unappendix


\end{document}